\documentclass[11pt]{article}
\usepackage{amssymb,amsmath,amsthm,multicol}
\usepackage[dvipdfmx]{graphicx}
\allowdisplaybreaks

\newtheorem{theorem}{Theorem}[]
\newtheorem{lemma}{Lemma}[]
\newtheorem*{lemma*}{Lemma}
\newtheorem{proposition}{Proposition}[]

\theoremstyle{definition}
\newtheorem{definition}{Definition}[]
\newtheorem{remark}{Remark}[]

\newcommand{\bldL}[1]{\mbox{\boldmath $#1$}}
\newcommand{\bld}[1]{\mbox{\boldmath \scriptsize{$#1$}}}

\makeatletter
 \def\eqnarray{%
 \stepcounter{equation}%
 \let\@currentlabel=\theequation
 \global\@eqnswtrue
 \global\@eqcnt\z@
 \tabskip\@centering
 \let\\=\@eqncr
 $$\halign to \displaywidth\bgroup\@eqnsel\hskip\@centering
 $\displaystyle\tabskip\z@{##}$&\global\@eqcnt\@ne
 \hfil$\displaystyle{{}##{}}$\hfil
 &\global\@eqcnt\tw@$\displaystyle\tabskip\z@{##}$\hfil
 \tabskip\@centering&\llap{##}\tabskip\z@\cr}
\makeatother

\title{On a distribution function of 
a probability measure involving a permutation}
\author{Yuichi KAMIYA \\
        Faculty of Economics \\
        Daito-Bunka University \\
        560 Iwadono, Higashi-Matsuyama \\
        Saitama 355-8501, Japan \\
        Email: ykamiya@ic.daito.ac.jp
   \and Tatsuya OKADA \\
        Department of Natural Sciences \\
        Fukushima Medical University \\
        Fukushima 960-1295, Japan \\
        Email: tokada@fmu.ac.jp
   \and Takeshi SEKIGUCHI \\
        Department of Infomation Science \\
        Faculty of Liberal Arts and Science \\
        Tohoku Gakuin University \\
        Sendai 981-3193, Japan \\
        Email: johmo@jcom.home.ne.jp
   \and Yasunobu SHIOTA \\ 
        Department of Infomation Science \\
        Faculty of Liberal Arts and Science \\
        Tohoku Gakuin University \\
        Sendai 981-3193, Japan \\
        Email: shiota@cs.tohoku-gakuin.ac.jp
       }
\date{\today}
\begin{document}
\maketitle
\begin{abstract}
In \cite{KOSS}, we have introduced a probability measure 
to study the power and exponential sums for a certain coding system.  
The distribution function of the probability measure gives 
explicit formulas for the power and exponential sums.  

\cite[Theorem 4]{KOSS} states that the higher order derivatives  
of the distribution function with respect to a certain parameter 
are expressed by a generalization of the Takagi function. 
In \cite{KOSS}, we only gave the sketch of the proof of Theorem 4, 
because the complete proof is very long. 
The purpose of this paper is to give the complete proof of 
\cite[Theorem 4]{KOSS}. 
\end{abstract}
\section{Introduction}\label{pm} 
Let $q\ge 2$ be an integer and $\sigma$ be a permutation
\[
\sigma= 
\Big(
\begin{array}{cccc}
0 & 1 & \cdots & q-1 \\   
\sigma(0) & \sigma(1) & \cdots & \sigma(q-1)
\end{array}
\Big).
\]
Throughout the paper, we assume that $\sigma^{q}={\rm id}$. 
A probability measure involving $\sigma$ on the unit interval
has been introduced in \cite{KOSS}. 
Let us recall the definition of the measure.

Let $I=I_{0}(0)=[0,1]$, and for each positive integer $k$, let  
\begin{eqnarray*}
  I_{k}(n) &=& \bigl[\frac{n}{q^k},\frac{n+1}{q^k}\bigl),\ \ 
                0 \leq n \leq q^{k}-2, \\
  I_{k}(q^{k}-1) &=& \bigl[\frac{q^{k}-1}{q^k},1\bigl]. 
\end{eqnarray*}
We denote the $\sigma$-field $\sigma\{I_{k}(n); 0 \leq n \leq q^{k}-1\}$ 
by $\mathcal{F}_{k}$ and the $\sigma$-field 
$\vee_{k=0}^{\infty}\mathcal{F}_{k}$ by ${\mathcal F}$. 

\begin{definition}\label{measure}
Let ${\bldL d}=(d_{0},\ldots ,d_{q-2})$ be a vector with 
$0<d_{j}<1$ $(0 \leq j \leq q-2)$ and $0<\sum_{j=0}^{q-2}d_{j}<1$,  
and set $d_{q-1}=1-\sum_{j=0}^{q-2}d_{j}$. 
Let ${\bldL r}=(r_{0},\ldots ,r_{q-2})$ be a vector whose components 
satisfy the same conditions as those of ${\bldL d}$, and 
set $r_{q-1}=1-\sum_{j=0}^{q-2}r_{j}$. 
Then the probability measure 
$\mu_{{\bld d},{\bld r}}$ involving a permutation $\sigma$
on $(I,\mathcal{F})$ is defined as follows. 

\noindent
{\rm (i)} $\mu_{{\bld d},{\bld r}}(I)=1$,

\noindent
{\rm (ii)} $\mu_{{\bld d},{\bld r}}(I_{1}(n))=d_{n}$, 
\quad $0\le n\le q-1$,

\noindent
{\rm (iii)} for $k\ge 2$, 
$$
\mu_{{\bld d},{\bld r}}(I_{k}(n))
=\mu_{{\bld d},{\bld r}}(I_{k-1}(j))\times r_{\sigma^{j}(l)},
\quad 0\le n\le q^{k}-1,
$$
\noindent
where $j$ and $l$ are integers with $n=qj+l$ $(0\le l\le q-1)$. 
The distribution function $L_{{\bld d},{\bld r}}$ of 
$\mu_{{\bld d},{\bld r}}$ is defined by 
\[
  L_{{\bld d},{\bld r}}(x)=\mu_{{\bld d},{\bld r}}([0,x]),\quad x\in I.
\]
For simplicity, we use the abbreviation $L_{\bld r}$
for $L_{{\bld r},{\bld r}}$.
\end{definition}

The measure $\mu_{{\bld d},{\bld r}}$
is a generalization of the multinomial measure (see \cite{O-S-S}) 
and the Gray measure (see \cite{Kob}). 

There is an interesting relation between $L_{\bld r}(x)$ 
and the exponential sum for a certain coding system 
related to paperfolding sequences (see \cite[Theorem 1]{KOSS}). 
Moreover, since $L_{\bld r}(x)$ is an analytic function of
${\bldL r}$ (see \cite[Theorem 2]{KOSS}), 
the power sums for the coding system are related to 
the higher order derivatives of $L_{\bld r}(x)$ with respect to 
${\bldL r}$ (see \cite[Theorem 3]{KOSS}). 

\cite[Theorem 4]{KOSS} states that the higher order derivatives  
of $L_{\bld r}(x)$ with respect to ${\bldL r}$ 
are expressed by a generalization of the Takagi function. 
To describe \cite[Theorem 4]{KOSS}, we prepare several notations.
Let ${\bldL q}$, ${\bldL e}_{l}$, and ${\bldL u}$ be vectors with 
\begin{eqnarray*}
  {\bldL q}    &=&(\underbrace{1/q,\ldots,1/q}_{q-1}), \\ 
  {\bldL e}_{l}&=&(\underbrace{0,\ldots,0,\overset{l}1,0\ldots,0}_{q-1}),
  \quad 0\le l\le q-2, \\
  {\bldL u}    &=& (u_{0},\ldots,u_{q-2}),\quad u_{l}\in{\bf N}\cup\{0\}, 
\end{eqnarray*}
and define 
\[
  |{\bldL u}|=u_{0}+u_{1}+\cdots+u_{q-2}, \quad
  {\bldL u}!=\prod_{l=0}^{q-2}u_{l}!. 
\]
For $n\in {\bf N}\cup\{0\}$, let ${\bldL r}_{\sigma^{n}}$ 
be the vector with 
\[
  {\bldL r}_{\sigma^{n}}=(r_{\sigma^{n}(0)},\ldots ,r_{\sigma^{n}(q-2)}). 
\]
For a set $S$, let ${\bf 1}_{S}$ be the indicator function of $S$.
Define the function $\Phi_{l}$ on $I$ by 
\[
\Phi_{l}=\sum_{j=0}^{q-1}{\bf 1}_{I_{2}(qj+\sigma^{-j}(l))},
\quad 0 \le l\le q-1. 
\]
Let $\phi(x)$ be the function on $I$ such that
$\phi(x)=qx \pmod{1}$ with $0\le\phi(x)<1$ for $x\in[0,1)$ 
and $\phi(1)=1$. We use the notation 
\[
 f\circ\phi^{j}(x)=f(\underbrace{\phi(\phi(\cdots\phi}_{j}(x))))
\]
for any function $f$. 
We denote the Lebesgue measure on $I$ by $\mu$.

\begin{definition} \label{gTakagi}
The generalized Takagi function  
${\mathcal T}_{{\bld d},{\bld r},{\bld u}}(x)$ is defined as follows. 

\noindent
(i) If ${\bldL u}={\bldL e}_{l}$, then 
\begin{align*}
{\mathcal T}_{{\bld d},{\bld r},{\bld e}_{l}}(x)
  =\frac{1}{q}\sum_{j=0}^{\infty}\sum_{n=0}^{q^{j}-1}
   \mu_{{\bld d},{\bld r}}(I_{j}(n)){\bf 1}_{I_{j}(n)}(x)
   \int_{0}^{\phi^{j}(x)}\Big(\frac{\Phi_{l}}{r_{l}}
  -\frac{\Phi_{q-1}}{r_{q-1}}\Big)d\mu_{{\bld r}_{\sigma^{n}},{\bld r}}. 
\end{align*}

\noindent
(ii) If $|{\bldL u}|\ge 2$, then
\begin{align*}
{\mathcal T}_{{\bld d},{\bld r},{\bld u}}(x)
  =& \sum_{j=0}^{\infty}
     \sum_{\stackrel{\scriptstyle{\alpha=0}}{u_{\alpha}>0}}^{q-2}
     \Big(\Big(\frac{\Phi_{\alpha}}{r_{\alpha}}
    -\frac{\Phi_{q-1}}{r_{q-1}}\Big)\circ\phi^{j}(x)\Big)\\
   & \times\sum_{n=0}^{q^{j+1}-1}
     \mu_{{\bld d},{\bld r}}(I_{j+1}(n)){\bf 1}_{I_{j+1}(n)}(x)
     \Big(
     {\mathcal T}_{{\bld r}_{\sigma^{n}},{\bld r},{\bld u}-{\bld e}_{\alpha}}
     \circ\phi^{j+1}(x)\Big). 
\end{align*}
\end{definition}

Then the higher order derivatives of $L_{\bld r}(x)$ with respect to 
${\bldL r}$ are expressed as the following. 

\begin{theorem} {\rm (\cite[Theorem 4]{KOSS})} \label{takagiTH}
{\rm (i)} If ${\bldL u}={\bldL e}_{l}$, then
\begin{eqnarray*}
  \frac{1}{q}\frac{\partial}{\partial r_{l}}L_{\bld r}(x)
   &=&({\bf 1}_{I_{1}(l)}(x) - {\bf 1}_{I_{1}(q-1)}(x))
      (L_{{\bld q},{\bld r}}(x)-x)\\ 
   & &+ \Big(\sum_{n=0}^{q-1}r_{n}{\bf 1}_{I_{1}(n)}(x)\Big)
             q{\mathcal T}_{{\bld q},{\bld r},{\bld e}_{l}}(x)
     +\int_{0}^{x}({\bf 1}_{I_{1}(l)}-{\bf 1}_{I_{1}(q-1)})d\mu.
\end{eqnarray*}

\noindent
{\rm (ii)} If $|{\bldL u}|\ge 2$, then
\begin{eqnarray*}
  \frac{1}{q{\bldL u}!}
  \frac{\partial^{u_{0}+\cdots +u_{q-2}}}
       {\partial r_{0}^{u_{0}}\cdots \partial r_{q-2}^{u_{q-2}}}
       L_{\bld r}(x)
   &=& \sum_{\stackrel{\scriptstyle{j=0}}{u_{j}>0}}^{q-2}
     ({\bf 1}_{I_{1}(j)}(x) - {\bf 1}_{I_{1}(q-1)}(x))
     q{\mathcal T}_{{\bld q},{\bld r},{\bld u}-{\bld e}_{j}}(x) \\
   & &+ \Big(\sum_{n=0}^{q-1}r_{n}{\bf 1}_{I_{1}(n)}(x)\Big)
     q{\mathcal T}_{{\bld q},{\bld r},{\bld u}}(x).
\end{eqnarray*}
\end{theorem}

In \cite{KOSS}, we only gave the sketch of the proof of the above 
Theorem 1, because the complete proof is very long. 
The purpose of this paper is to give the complete proof of Theorem 1.

Finally, we mention the previous works on studying relations 
between higher order derivatives of distribution functions 
and Takagi functions. 
Hata--Yamaguti \cite{HY} is the first work clarifying  
a relation between the first order derivative of $L_{\bld r}(x)$ 
with respect to ${\bldL r}$ and the usual Takagi function 
in the dyadic case. 
In \cite{S-S}, Hata--Yamaguti's result is studied 
from a viewpoint of the binomial measure, 
and, in \cite{O-S-S},
it is generalized in the $q$-adic case, 
in which the multinomial measure 
and its distribution function play essential roles. 
In Kobayashi \cite{Kob}, the Gray measure and its distribution function 
are studied from a viewpoint of \cite{S-S} and \cite{O-S-S}. 
Since the measure $\mu_{{\bld d},{\bld r}}$
is a generalization of the multinomial measure and Gray measure, 
Theorem 1 is a natural generalization of the results obtained in 
\cite{HY}, \cite{S-S}, \cite{O-S-S}, and \cite{Kob}.

\section{Preliminary lemmas}

For a fixed $k\in{\bf N}$, any integer $n$ with $0\le n\le q^{k}-1$ 
is expressed as
$n=\sum_{i=0}^{k-1}n_{i}q^{i}$, where 
$n_{i}\in\{0,1,\ldots q-1\}$. We use the abbreviation 
$n=n_{k-1}\cdots n_{0}$ for $n=\sum_{i=0}^{k-1}n_{i}q^{i}$,
in which the length of the word is always equal to $k$,
and identify $I_{k}(n)$ with $I_{k}(n_{k-1}\cdots n_{0})$.  

Firstly, we study a relation between $\phi^{i}$ 
and $\mu_{{\bld d},{\bld r}}$. 
We note a simple fact
\begin{align*}
I_{i+k}(b_{i-1}\cdots b_{0}a_{k-1}\cdots a_{0})
\subset I_{i}(c_{i-1}\cdots c_{0})
\Leftrightarrow
b_{i-1}\cdots b_{0}=c_{i-1}\cdots c_{0}.
\end{align*}

\begin{lemma}\label{phiaction}
We have  
\[
\phi^{i}\Big(\bigcup_{0 \leq  b_{0}, \cdots, b_{i-1} \leq q-1} 
I_{i+k}(b_{i-1}\cdots b_{0}a_{k-1}\cdots a_{0})\Big)
=I_{k}(a_{k-1}\cdots a_{0}).
\]
\end{lemma}
\begin{proof}
By the definition of $\phi$, we have 
\begin{eqnarray*}
 & & \phi\Big(
         \bigcup_{0 \leq  b_{0}, \cdots, b_{i-1} \leq q-1}
            I_{i+k}(b_{i-1}\cdots b_{0}a_{k-1}\cdots a_{0})
      \Big)\\
 & & = \bigcup_{0 \leq  b_{0}, \cdots, b_{i-2} \leq q-1}
       I_{i-1+k}(b_{i-2}\cdots b_{0}a_{k-1}\cdots a_{0}).
\end{eqnarray*}
Repeating this $i$ times, we obtain the assertion. 
\end{proof}

Lemma \ref{phiaction} is equivalent to the following. 

\begin{lemma}\label{phicom}
We have
\[
 {\bf 1}_{I_{k}(a_{k-1}\cdots a_{0})}\circ\phi^{i}
  ={\bf 1}_{\bigcup_{0 \leq b_{0},\cdots, b_{i-1} \leq q-1}
    I_{i+k}(b_{i-1}\cdots b_{0}a_{k-1}\cdots a_{0})}.
\]
\end{lemma}

\begin{lemma}\label{phimeasure}
We have
\[
\mu_{{\bld d},{\bld r}}(I_{i+k}(b_{i-1}\cdots b_{0}a_{k-1}\cdots a_{0}))
=\mu_{{\bld d},{\bld r}}(I_{i}(b_{i-1}\cdots b_{0}))
\mu_{{\bld r}_{\sigma^{b_{0}}},{\bld r}}(I_{k}(a_{k-1}\cdots a_{0})).
\]
\end{lemma}
\begin{proof}
It follows from Definition \ref{measure} and 
the assumption $\sigma^{q}={\rm id}$ that 
\begin{eqnarray*}
& &\mu_{{\bld d},{\bld r}}(I_{i+k}(b_{i-1}\cdots b_{0}a_{k-1} 
\cdots a_{0}))
=\mu_{{\bld d},{\bld r}}(I_{i}(b_{i-1}\cdots b_{0}))
r_{\sigma^{b_{0}}(a_{k-1})}r_{\sigma^{a_{k-1}}(a_{k-2})}
\cdots r_{\sigma^{a_{1}}(a_{0})} 
\end{eqnarray*}
and
\begin{eqnarray*}
\mu_{{\bld r}_{\sigma^{b_{0}}},{\bld r}}(I_{k}(a_{k-1}\cdots a_{0}))
&=& r_{\sigma^{b_{0}}(a_{k-1})}r_{\sigma^{a_{k-1}}(a_{k-2})} 
\cdots r_{\sigma^{a_{1}}(a_{0})}.
\end{eqnarray*}
Hence we obtain the assertion. 
\end{proof}

\begin{lemma}\label{phix}
For any $i\in {\bf N}$, $a \in {\bf N}\cup\{0\}$ 
with $0 \leq a \leq q^{i}-1$, and 
$x \in I_{i}(a_{i-1}\cdots a_{0})$, we have 
\begin{equation}\label{pre22-2}
  {\bf 1}_{[0,\phi^{i}(x)]}(\phi^{i}(y))\times 
  {\bf 1}_{I_{i}(a_{i-1}\cdots a_{0})}(y)
 ={\bf 1}_{[0,x]}(y)\times
  {\bf 1}_{I_{i}(a_{i-1}\cdots a_{0})}(y).
\end{equation}
\end{lemma}
\begin{proof}
We prove this by induction on $i$. 
When $i=1$, we have for $x \in I_{1}(a_{0})$
$$
\phi(y)\in[0,\phi(x)]\Leftrightarrow
y\in\bigcup_{m=0}^{q-1}\Big[\frac{m}{q},\frac{m}{q}+|x-\frac{a_{0}}{q}|\Big],
$$
and hence 
$$
\phi(y)\in[0,\phi(x)]\ \mbox{and}\ y\in I_{1}(a_{0})
\Leftrightarrow y \in [0,x] \cap I_{1}(a_{0}),
$$
from which we get
\begin{equation}\label{pre22-3}
{\bf 1}_{[0,\phi(x)]}(\phi(y))\times 
{\bf 1}_{I_{1}(a_{0})}(y)
={\bf 1}_{[0,x]}(y)\times 
{\bf 1}_{I_{1}(a_{0})}(y).
\end{equation}

\noindent
By Lemma \ref{phicom}, we have
\begin{equation}\label{pre23}
{\bf 1}_{I_{1}(a_{0})}(\phi^{i}(y))\times
{\bf 1}_{I_{i}(a_{i}\cdots a_{1})}(y)
={\bf 1}_{I_{i+1}(a_{i}\cdots a_{1}a_{0})}(y).
\end{equation}
Therefore, if $x \in I_{i+1}(a_{i}\cdots a_{1}a_{0})$ and 
(\ref{pre22-2}) holds for $i$, by \eqref{pre23}, \eqref{pre22-3}, 
and  Lemma \ref{phicom}, we have
\begin{eqnarray*}
& &{\bf 1}_{[0,\phi^{i+1}(x)]}(\phi^{i+1}(y))\times 
{\bf 1}_{I_{i+1}(a_{i}\cdots a_{1}a_{0})}(y)\\  
& &= {\bf 1}_{[0,\phi(\phi^{i}(x))]}(\phi(\phi^{i}(y)))\times 
{\bf 1}_{I_{1}(a_{0})}(\phi^{i}(y))\times 
{\bf 1}_{I_{i}(a_{i}\cdots a_{1})}(y)\\
& &= {\bf 1}_{[0,\phi^{i}(x)]}(\phi^{i}(y))\times 
{\bf 1}_{I_{1}(a_{0})}(\phi^{i}(y))\times 
{\bf 1}_{I_{i}(a_{i}\cdots a_{1})}(y)\\
& &= {\bf 1}_{[0,x]}(y)\times {\bf 1}_{I_{i}(a_{i}\cdots a_{1})}(y)
\times {\bf 1}_{I_{1}(a_{0})}(\phi^{i}(y))\\
& &= {\bf 1}_{[0,x]}(y)\times 
{\bf 1}_{I_{i+1}(a_{i}\cdots a_{1}a_{0})}(y). 
\end{eqnarray*}
This completes the proof.
\end{proof}

For any bounded $\mathcal{F}$-measurable function $f$, let
\begin{align*}
{\rm E}_{\mu_{{\bld d},{\bld r}}}(f)
& =\int_{I}fd\mu_{{\bld d},{\bld r}},\\
{\rm E}_{\mu_{{\bld d},{\bld r}}}(f;I_{k}(n))
& =\int_{I_{k}(n)}fd\mu_{{\bld d},{\bld r}}.
\end{align*}
Lemmas \ref{phiE} and \ref{phiEx} show that 
a kind of integration by substitution is valid.

\begin{lemma}\label{phiE}
For any $i\in {\bf N}$, $a\in {\bf N}\cup\{0\}$ 
with $0 \leq a \leq q^{i}-1$, 
and a bounded ${\mathcal F}$-measurable function $f$, we have
$$
{\rm E}_{\mu_{{\bld d},{\bld r}}}(f\circ\phi^{i};I_{i}(a_{i-1}\cdots a_{0}))
=\mu_{{\bld d},{\bld r}}(I_{i}(a_{i-1}\cdots a_{0}))
{\rm E}_{\mu_{{\bld r}_{\sigma^{a_{0}}},{\bld r}}}(f).
$$
\end{lemma}
\noindent{\it Proof.}
Since a bounded ${\mathcal F}$-measurable function can be 
approximated by step functions, it suffices to show the equality for 
$f={\bf 1}_{I_{j}(c_{j-1}\cdots c_{0})}$.
By Lemmas \ref{phicom} and \ref{phimeasure}, we have
\begin{align*}
& {\rm E}_{\mu_{{\bld d},{\bld r}}}({\bf 1}_{I_{j}(c_{j-1}\cdots c_{0})} 
\circ\phi^{i};I_{i}(a_{i-1}\cdots a_{0}))\\
& = \int_I{\bf 1}_{\bigcup_{0\le b_{0},\ldots,b_{i-1}\le q-1}
I_{i+j}(b_{i-1}\cdots b_{0}c_{j-1}\cdots c_{0})}\times
{\bf 1}_{I_{i}(a_{i-1}\cdots a_{0})}d\mu_{{\bld d},{\bld r}}\\
& =\int_I{\bf 1}_{I_{i+j}(a_{i-1}\cdots a_{0}c_{j-1}\cdots c_{0})}
d\mu_{{\bld d},{\bld r}}\\
& =\mu_{{\bld d},{\bld r}}(I_{i+j}
(a_{i-1}\cdots a_{0}c_{i-1}\cdots c_{0}))\\
& =\mu_{{\bld d},{\bld r}}(I_{i}(a_{i-1}\cdots a_{0}))
{\rm E}_{\mu_{{\bld r}_{\sigma^{a_{0}}},{\bld r}}}({\bf 1}_{I_{j}(c_{j-1}\cdots c_{0})}).
\end{align*}

\vspace{-7mm}\qed

\vspace{3mm}

\begin{lemma}\label{phiEx}
For any $i\in {\bf N}$, $a\in {\bf N}\cup\{0\}$ 
with $0 \leq a \leq q^{i}-1$, 
a bounded ${\mathcal F}$-measurable function $f$, 
and $x \in I_{i}(a_{i-1}\cdots a_{0})$, we have
$$
{\rm E}_{\mu_{{\bld d},{\bld r}}}(f\circ\phi^{i};I_{i}(a_{i-1}\cdots a_{0})\cap[0,x])
=\mu_{{\bld d},{\bld r}}(I_{i}(a_{i-1}\cdots a_{0}))
{\rm E}_{\mu_{{\bld r}_{\sigma^{a_{0}}},{\bld r}}}(f;[0,\phi^{i}(x)]).
$$
\end{lemma}
\noindent{\it Proof.}
By Lemmas \ref{phix} and \ref{phiE}, we obtain
\begin{eqnarray*}
& & {\rm E}_{\mu_{{\bld d},{\bld r}}}(f\circ\phi^{i};I_{i}(a_{i-1}\cdots a_{0})\cap[0,x])\\
& &= {\rm E}_{\mu_{{\bld d},{\bld r}}}((f\circ\phi^{i})\times
{\bf 1}_{[0,x]};I_{i}(a_{i-1}\cdots a_{0}))\\
& &={\rm E}_{\mu_{{\bld d},{\bld r}}}((f\circ\phi^{i})\times
({\bf 1}_{[0,\phi^{i}(x)]}\circ\phi^{i});
I_{i}(a_{i-1}\cdots a_{0}))\\
& &=\mu_{{\bld d},{\bld r}}(I_{i}(a_{i-1}\cdots a_{0}))
{\rm E}_{\mu_{{\bld r}_{\sigma^{a_{0}}},{\bld r}}}
(f\times{\bf 1}_{[0,\phi^{i}(x)]}).
\end{eqnarray*}

\vspace{-7mm}\qed

\vspace{3mm}

Next, we discuss the conditional expectation 
${\rm E}_{\mu_{{\bld d},{\bld r}}}(\,\cdot\,|{\mathcal F}_{k})$. 
For a bounded $\mathcal{F}$-measurable function $g$,
${\rm E}_{\mu_{{\bld d},{\bld r}}}(g|{\mathcal F}_{k})$ is 
defined to be the ${\mathcal F}_{k}$-measurable function such that
$$
\int_{G}{\rm E}_{\mu_{{\bld d},{\bld r}}}(g|{\mathcal F}_{k})\,d\mu_{{\bld d},{\bld r}}
=\int_{G}g\,d\mu_{{\bld d},{\bld r}},\quad 
\mbox{for all $G\in{\mathcal F}_{k}$}.
$$
Since ${\mathcal F}_{k}$ is the finite set and 
${\rm E}_{\mu_{{\bld d},{\bld r}}}(g|{\mathcal F}_{k})$ is 
${\mathcal F}_{k}$-measurable, 
${\rm E}_{\mu_{{\bld d},{\bld r}}}(g|{\mathcal F}_{k})$
is a step function with constant values on $I_{k}(n)$'s.
In fact, it is written explicitely as 
\begin{equation}\label{kinji00}
{\rm E}_{\mu_{{\bld d},{\bld r}}}(g|{\mathcal F}_{k})
=\sum_{n=0}^{q^{k}-1}
\frac{{\rm E}_{\mu_{{\bld d},{\bld r}}}(g;I_{k}(n))}{\mu_{{\bld d},{\bld r}}(I_{k}(n))}{\bf 1}_{I_{k}(n)}. 
\end{equation}

\begin{lemma}\label{bubun}
Let $g$ be a bounded $\mathcal{F}$-measurable function. 
If $h$ is a ${\mathcal F}_{k}$-measurable function, and
$g$ satisfies ${\rm E}_{\mu_{{\bld d},{\bld r}}}(g|{\mathcal F}_{k})=0$,
then
$$
\int_{0}^{x}hg\,d\mu_{{\bld d},{\bld r}}
=h(x)\int_{0}^{x}g\,d\mu_{{\bld d},{\bld r}}.
$$
\end{lemma}

\begin{proof}
By \eqref{kinji00}, 
${\rm E}_{\mu_{{\bld d},{\bld r}}}(g|{\mathcal F}_{k})=0$
is equivalent to ${\rm E}_{\mu_{{\bld d},{\bld r}}}(g;I_{k}(n))=0$
for every $n$. 
Since $h$ is ${\mathcal F}_{k}$-measurable,
it takes a constant value $C_{n}$ on $I_{k}(n)$. Hence
\begin{align*}
& {\rm E}_{\mu_{{\bld d},{\bld r}}}(hg;I_{k}(n))
=C_{n}{\rm E}_{\mu_{{\bld d},{\bld r}}}(g;I_{k}(n))=0.
\end{align*}
Thus we obtain for $x \in I_{k}(m)$ 
\begin{align*}
\int_{0}^{x}hg\,d_{\mu_{{\bld d},{\bld r}}}
& = {\rm E}_{\mu_{{\bld d},{\bld r}}}(hg;I_{k}(m)\cap[0,x])\\
& = h(x){\rm E}_{\mu_{{\bld d},{\bld r}}}(g;I_{k}(m)\cap[0,x])
= h(x)\int_{0}^{x}g\,d\mu_{{\bld d},{\bld r}}.
\end{align*}
Since the equality is independent of $m$, it is valid for $x\in I$.
\end{proof}

\section{The Radon-Nikodym derivative on the finite set}

Let ${\bldL e}=(e_{0},\ldots ,e_{q-2})$ and 
${\bldL s}=(s_{0},\ldots ,s_{q-2})$ 
be vectors whose components satisfy 
the same conditions as those of 
${\bldL d}$ in Definition \ref{measure},  
and set $e_{q-1}=1-\sum_{j=0}^{q-2}e_{j}$ and  $s_{q-1}=1-\sum_{j=0}^{q-2}s_{j}$.

\begin{definition}\label{Zdef}
The function 
$Z\Big[
  \begin{array}{ll}
    {\bldL e} & {\bldL s} \\
    {\bldL d} & {\bldL r} 
  \end{array}
  ;k\Big]:I\to{\bf R}$ 
is defined by
\begin{equation*}
  Z\Big[
    \begin{array}{ll}
      {\bldL e} & {\bldL s} \\
      {\bldL d} & {\bldL r} 
    \end{array}
  ;k\Big]
  =\sum_{n=0}^{q^{k}-1}
     \frac{\mu_{{\bld e},{\bld s}}(I_{k}(n))}
          {\mu_{{\bld d},{\bld r}}(I_{k}(n))}
     {\bf 1}_{I_{k}(n)},\quad k \in {\bf N}\cup\{0\}. 
\end{equation*}
\end{definition}

\begin{remark}
$Z\Big[
  \begin{array}{ll}
    {\bldL e} & {\bldL s} \\
    {\bldL d} & {\bldL r} 
  \end{array}
  ;k\Big]$ 
is the so-called Radon-Nikodym derivative 
$d\mu_{{\bld e},{\bld s}}/d\mu_{{\bld d},{\bld r}}$ on $\mathcal{F}_{k}$.
\end{remark}

We identify $I_{k}(n)$ with $I_{k}(n_{k-1}\cdots n_{0})$
as in the previous section.  

\begin{definition}\label{Wdef}
The function $W\Big[
\begin{array}{l}
 {\bldL s} \\
 {\bldL r} 
\end{array}\Big]:I\to{\bf R}$ is defined by
$$
W\Big[
\begin{array}{l}
 {\bldL s} \\
 {\bldL r} 
\end{array}\Big]
=\sum_{0\le b_{0}, b_{1}\le q-1}
\frac{s_{\sigma^{b_{1}}(b_{0})}}{r_{\sigma^{b_{1}}(b_{0})}}
{\bf 1}_{I_{2}(b_{1}b_{0})}.
$$
\end{definition}

The following propositions have been proved in \cite{KOSS}.

\begin{proposition}\label{weak}
We have 
\[
  L_{{\bld e},{\bld s}}(x)
    =\lim_{k\to\infty}\int_{0}^{x}
       Z\Big[
         \begin{array}{ll}
          {\bldL e} & {\bldL s} \\
          {\bldL d} & {\bldL r} 
         \end{array}
       ;k\Big]d\mu_{{\bld d},{\bld r}}, 
\]
where the convergence is uniform for 
${\bldL e}=(e_{0},\ldots,e_{q-2})$
and ${\bldL s}=(s_{0},\ldots,s_{q-2})$. 
\end{proposition}

\begin{proposition}\label{Wproduct}
For $k\ge 1$, we have
$$
Z\Big[
\begin{array}{ll}
{\bldL e} & {\bldL s} \\
{\bldL d} & {\bldL r} 
\end{array}
;k+1\Big]
=
\Big(
\prod_{i=0}^{k-1}
W\Big[
\begin{array}{l}
 {\bldL s} \\
 {\bldL r} 
\end{array}\Big]
\circ \phi^{i}
\Big)
Z\Big[
\begin{array}{ll}
{\bldL e} & {\bldL s} \\
{\bldL d} & {\bldL r} 
\end{array}
;1\Big].
$$
\end{proposition}


\section{Higher order derivatives of distribution functions}\label{HOD}

Firstly, we study a relation between
$L_{\bld s}(x)(=L_{{\bld s},{\bld s}}(x))$ and 
$L_{{\bld q},{\bld s}}(x)$.

\begin{lemma}\label{basechange}
We have
\begin{align*}
L_{\bld s}(x)=\Big(q\sum_{n=0}^{q-1}s_{n}{\bf 1}_{I_{1}(n)}(x)\Big)
(L_{{\bld q},{\bld s}}(x)-x)
+q\int_{0}^{x}\sum_{n=0}^{q-1}s_{n}{\bf 1}_{I_{1}(n)}d\mu, 
\end{align*}
where $\mu$ is the Lebesgue measure on $I$. 
\end{lemma}
\begin{proof}
Let $x\in  I_{1}(m)\ (0 \leq m \leq q-1)$. Then it follows that 
\begin{eqnarray}\label{6.2eq1}
L_{\bld s}(x)&=&
L_{\bld s}\Big(\frac{m}{q}\Big)+\mu_{{\bld s},{\bld s}}
\Big(\Big(\frac{m}{q},x\Big]\Big)\nonumber\\
&=&\sum_{n=0}^{m-1}s_{n}+\frac{s_{m}}{1/q}\mu_{{\bld q},{\bld s}}\Big(\Big(\frac{m}{q},x\Big]\Big)\nonumber\\
&=&q s_{m}(L_{{\bld q},{\bld s}}(x)-x)+\sum_{n=0}^{m-1}s_{n}+q s_{m}\Big(x-\frac{m}{q}\Big).
\end{eqnarray}
Noting that, for $x\in  I_{1}(m)$,
\[
\int_{0}^{x}{\bf 1}_{I_{1}(n)}d\mu=
\begin{cases}
0, &   n>m,\\          
x-\frac{n}{q}, & n=m, \\   
\frac{1}{q}, & n<m,
\end{cases}
\]
we have
\begin{equation}\label{6.2eq2}
\sum_{n=0}^{m-1}s_{n}+qs_{m}\Big(x-\frac{m}{q}\Big)
=\sum_{n=0}^{q-1}qs_{n}\int_{0}^{x}{\bf 1}_{I_{1}(n)}d\mu.
\end{equation}
Substituting \eqref{6.2eq2} into \eqref{6.2eq1} 
and replacing the range of the variable $x$ to $I$, 
we obtain the assertion.
\end{proof}

By Lemma \ref{basechange},
we have easily the following relation between 
the higher order derivative of $L_{\bld s}(x)$ and
that of $L_{{\bld q},{\bld s}}(x)$.

\begin{lemma}\label{HDchange2} 
{\rm (i)} If ${\bldL u}={\bldL e}_{l}$, then
\begin{align*}
\frac{1}{q}\frac{\partial}{\partial s_{l}}L_{\bld s}(x)
=&\ ({\bf 1}_{I_{1}(l)}(x)-{\bf 1}_{I_{1}(q-1)}(x))
(L_{{\bld q},{\bld s}}(x)-x)\\
& +\Big(\sum_{n=0}^{q-1}s_{n}{\bf 1}_{I_{1}(n)}(x)\Big)
\frac{\partial}{\partial s_{l}}L_{{\bld q},{\bld s}}(x)
+\int_{0}^{x}({\bf 1}_{I_{1}(l)}-{\bf 1}_{I_{1}(q-1)})d\mu.
\end{align*}
{\rm (ii)} If $|{\bldL u}|\ge 2$, then
\begin{align*}
\frac{1}{q}\frac{\partial^{u_{0}+\cdots +u_{q-2}}}
{\partial s_{0}^{u_{0}}\cdots \partial s_{q-2}^{u_{q-2}}}
L_{\bld s}(x)
= &\ \sum_{\stackrel{\scriptstyle{j=0}}{u_{j}>0}}^{q-2}
u_{j}({\bf 1}_{I_{1}(j)}(x)-{\bf 1}_{I_{1}(q-1)}(x))\\
& \times\frac{\partial^{u_{0}+\cdots +u_{j-1}+(u_{j}-1)+u_{j+1}+\cdots+u_{q-2}}}
{\partial s_{0}^{u_{0}}\cdots\partial s_{j-1}^{u_{j-1}}
\partial s_{j}^{u_{j}-1}\partial s_{j+1}^{u_{j+1}}
\cdots \partial s_{q-2}^{u_{q-2}}}
L_{{\bld q},{\bld s}}(x)\\
& +\Big(\sum_{n=0}^{q-1}s_{n}{\bf 1}_{I_{1}(n)}(x)\Big)
\frac{\partial^{u_{0}+\cdots +u_{q-2}}}
{\partial s_{0}^{u_{0}}\cdots \partial s_{q-2}^{u_{q-2}}}
L_{{\bld q},{\bld s}}(x).
\end{align*}
\end{lemma}

Next, we study the higher order derivative of $L_{{\bld q},{\bld s}}(x)$.
Let $\psi_{\bld u}:\{1,2,\ldots,|{\bldL u}|\}\rightarrow
\{0,1,\ldots,q-2\}$ be a mapping such that 
$\#\{m;\psi_{{\bld u}}(m)=j\}=u_{j}$, 
which is the same one as that of \cite{O-S-S}. 
For example, if $q=4$, ${\bldL u}=(u_{0},u_{1},u_{2})=(1,2,0)$,
then $\psi_{\bld u}:\{1,2,3\}\rightarrow\{0,1,2\}$
is a mapping satisfying
$\#\{m;\psi_{{\bld u}}(m)=0\}=1$,
$\#\{m;\psi_{{\bld u}}(m)=1\}=2$,
and $\#\{m;\psi_{{\bld u}}(m)=2\}=0$. 
In fact, $\psi_{\bld u}$ is one of three mappings
$$
\begin{cases}
\psi_{\bld u}(1)=0, \\
\psi_{\bld u}(2)=1, \\
\psi_{\bld u}(3)=1, 
\end{cases}\quad
\begin{cases}
\psi_{\bld u}(1)=1, \\
\psi_{\bld u}(2)=0, \\
\psi_{\bld u}(3)=1, 
\end{cases}\quad
\begin{cases}
\psi_{\bld u}(1)=1, \\
\psi_{\bld u}(2)=1, \\
\psi_{\bld u}(3)=0.
\end{cases}
$$

\begin{lemma}\label{Hsrq}
We have 
\begin{align*}
& \frac{\partial^{u_{0}+\cdots+u_{q-2}}}
{\partial s_{0}^{u_{0}}\cdots 
\partial s_{q-2}^{u_{q-2}}}L_{{\bld q},{\bld s}}(x)
\Big|_{{\bld s}={\bld r}}\\ 
& ={\bldL u}!\lim_{k\to\infty}\sum_{0\le i_{1}<\cdots <i_{|\bld u|}\le k-2}\sum_{\psi_{\bld u}}
\int_{0}^{x}\prod_{m=1}^{|\bld u|}
\Big(\frac{\Phi_{\psi_{\bld u}(m)}}{r_{\psi_{\bld u}(m)}}
-\frac{\Phi_{q-1}}{r_{q-1}}\Big)\circ\phi^{i_{m}}
d\mu_{{\bld q},{\bld r}},
\end{align*}
where the sum $\sum_{\psi_{\bld u}}$ is taken over all $\psi_{\bld u}$'s.
\end{lemma}

\begin{proof}
By Propositions \ref{weak} and \ref{Wproduct} with 
${\bldL e}={\bldL d}={\bldL q}$, we have 
\begin{align}
L_{{\bld q},{\bld s}}(x)
=& \lim_{k\to\infty}
\int_{0}^{x}
\prod_{i=0}^{k-2}
W\Big[
\begin{array}{l}
 {\bldL s} \\
 {\bldL r} 
\end{array}\Big]
\circ \phi^{i}
d\mu_{{\bld q},{\bld r}}.\label{integral00}
\end{align}
From the definitions of $W\Big[
\begin{array}{l}
 {\bldL s} \\
 {\bldL r} 
\end{array}\Big]$ and $\Phi_{l}$, it follows that
\begin{align}\label{phili1}
W\Big[
\begin{array}{l}
 {\bldL s} \\
 {\bldL r} 
\end{array}\Big]
\circ \phi^{i}
=& \sum_{l=0}^{q-1}\frac{s_{l}}{r_{l}}(\Phi_{l}\circ\phi^{i}).
\end{align}
Let ${\mathcal K}_{{\bld s},i}
=\sum_{l=0}^{q-1}s_{l}(\Phi_{l}\circ\phi^{i})$.
We show the equality
\begin{align}\label{phili2}
\sum_{l=0}^{q-1}\frac{s_{l}}{r_{l}}(\Phi_{l}\circ\phi^{i}) 
=\frac{{\mathcal K}_{{\bld s},i}}{{\mathcal K}_{{\bld r},i}}.
\end{align}
Since $\sum_{l=0}^{q-1}\Phi_{l}=1$, it holds that 
$\sum_{l=0}^{q-1}\Phi_{l}\circ\phi^{i}=1$. 
For any $x\in I$, there exists a unique $m$ such that 
$\Phi_{m}\circ\phi^{i}(x)=1,\ \Phi_{l}\circ\phi^{i}(x)=0 \ (l\ne m)$, 
and hence, both of 
$\sum_{l=0}^{q-1}\frac{s_{l}}{r_{l}}(\Phi_{l}\circ\phi^{i}(x))$
and $\frac{{\mathcal K}_{{\bld s},i}}{{\mathcal K}_{{\bld r},i}}(x)$
are $\frac{s_{m}}{r_{m}}$.
Comibining \eqref{integral00}, \eqref{phili1}, and \eqref{phili2},
we have
\begin{align}\label{Lqs}
L_{{\bld q},{\bld s}}(x)
=& \lim_{k\to\infty}
\int_{0}^{x}
\prod_{i=0}^{k-2}
\frac{{\mathcal K}_{{\bld s},i}}{{\mathcal K}_{{\bld r},i}}
d\mu_{{\bld q},{\bld r}}.
\end{align}
For $a$ with $0\le a\le q-2$,
$$
\frac{\partial}{\partial s_{a}}\frac{{\mathcal K}_{{\bld s},i}}{{\mathcal K}_{{\bld r},i}}
=\frac{(\Phi_{a}-\Phi_{q-1})\circ\phi^{i}}
{{\mathcal K}_{{\bld r},i}}.
$$
By the same argument as in \cite[pp.459--460]{O-S-S}, 
\begin{align*}
& \frac{\partial^{u_{0}+\cdots+u_{q-2}}}
{\partial s_{0}^{u_{0}}\cdots \partial s_{q-2}^{u_{q-2}}}
\int_{0}^{x}\prod_{i=0}^{k-2}
\frac{{\mathcal K}_{{\bld s},i}}{{\mathcal K}_{{\bld r},i}}d\mu_{{\bld q},{\bld r}}\nonumber\\
& ={\bldL u}!\sum_{0\le i_{1}<\cdots <i_{|\bld u|}\le k-2}
\sum_{\psi_{\bld u}}\int_{0}^{x}
\Big(\prod_{m=1}^{|\bld u|}
\frac{(\Phi_{\psi_{\bld u}(m)}-\Phi_{q-1})\circ
\phi^{i_{m}}}
{{\mathcal K}_{{\bld s},i_{m}}}\Big)\times
\Big(\prod_{i=0}^{k-2}\frac{{\mathcal K}_{{\bld s},i}}
{{\mathcal K}_{{\bld r},i}}\Big)
d\mu_{{\bld q},{\bld r}}.
\end{align*}
Hence
\begin{align}
& \frac{\partial^{u_{0}+\cdots+u_{q-2}}}
{\partial s_{0}^{u_{0}}\cdots \partial s_{q-2}^{u_{q-2}}}
\int_{0}^{x}\prod_{i=0}^{k-2}
\frac{{\mathcal K}_{{\bld s},i}}{{\mathcal K}_{{\bld r},i}}d\mu_{{\bld q},{\bld r}}
\Big|_{{\bld s}={\bld r}}\nonumber\\
& ={\bldL u}!\sum_{0\le i_{1}<\cdots <i_{|\bld u|}\le k-2}\sum_{\psi_{\bld u}}
\int_{0}^{x}\prod_{m=1}^{|\bld u|}
\Big(\frac{\Phi_{\psi_{\bld u}(m)}}{r_{\psi_{\bld u}(m)}}
-\frac{\Phi_{q-1}}{r_{q-1}}\Big)\circ\phi^{i_{m}}
d\mu_{{\bld q},{\bld r}}.\label{partial}
\end{align}
From \eqref{Lqs} and \eqref{partial}, the assertion follows.
\end{proof}

By Lemmas \ref{HDchange2} and \ref{Hsrq}, and 
$u_{j}({\bldL u}-{\bldL e}_{j})!={\bldL u}!$, we obtain the following.

\begin{proposition}\label{HLsrq}
{\rm (i)} If ${\bldL u}={\bldL e}_{l}$, then
\begin{eqnarray*}
\frac{1}{q}\frac{\partial}{\partial s_{l}}
L_{\bld s}(x)\Big|_{{\bld s}={\bld r}}
&=& ({\bf 1}_{I_{1}(l)}(x)-{\bf 1}_{I_{1}(q-1)}(x))
(L_{{\bld q},{\bld r}}(x)-x)\\
& &+\Big(\sum_{n=0}^{q-1}r_{n}{\bf 1}_{I_{1}(n)}(x)\Big)
\lim_{k\to\infty}\sum_{0\le j\le k-2}\int_{0}^{x}
\Big(\frac{\Phi_{l}}{r_{l}}
-\frac{\Phi_{q-1}}{r_{q-1}}\Big)
\circ\phi^{j}d\mu_{{\bld q},{\bld r}}\\
& &+\int_{0}^{x}({\bf 1}_{I_{1}(l)}-{\bf 1}_{I_{1}(q-1)})
d\mu.
\end{eqnarray*}
{\rm (ii)} If $|{\bldL u}|\ge 2$, then
\begin{eqnarray*}
& & \frac{1}{q{\bldL u}!}\frac{\partial^{u_{0}+\cdots +u_{q-2}}}
{\partial s_{0}^{u_{0}}\cdots \partial s_{q-2}^{u_{q-2}}}
L_{\bld s}(x)\Big|_{{\bld s}={\bld r}}
=\sum_{\stackrel{\scriptstyle{j=0}}{u_{j}>0}}^{q-2}
({\bf 1}_{I_{1}(j)}(x)-{\bf 1}_{I_{1}(q-1)}(x))\\
& &\qquad\times\lim_{k\to\infty}
\sum_{0\le i_{1}<\cdots <i_{|\bld u|-1}\le k-2} \ \sum_{\psi_{{\bld u}-{\bld e}_{j}}}
\int_{0}^{x}\prod_{m=1}^{|\bld u|-1}
\Big(\frac{\Phi_{\psi_{{\bld u}-{\bld e}_{j}}(m)}}{r_{\psi_{{\bld u}-{\bld e}_{j}}(m)}}
-\frac{\Phi_{q-1}}{r_{q-1}}\Big)\circ\phi^{i_{m}}d\mu_{{\bld q},{\bld r}}\\
& &\quad+\Big(\sum_{n=0}^{q-1}r_{n}{\bf 1}_{I_{1}(n)}(x)\Big)
\lim_{k\to\infty}\sum_{0\le i_{1}<\cdots <i_{|\bld u|}\le k-2}\sum_{\psi_{\bld u}}
\int_{0}^{x}\prod_{m=1}^{|\bld u|}
\Big(\frac{\Phi_{\psi_{{\bld u}}(m)}}{r_{\psi_{{\bld u}}(m)}}
-\frac{\Phi_{q-1}}{r_{q-1}}\Big)
\circ\phi^{i_{m}}d\mu_{{\bld q},{\bld r}}.
\end{eqnarray*}
\end{proposition}

\section{A recursive relation for
${\mathcal D}_{{\bld d},{\bld r},{\bld u},k}(x)$}\label{Deftakagi}

Based on the expression of Proposition \ref{HLsrq},
we introduce the function
${\mathcal D}_{{\bld d},{\bld r},{\bld u},k}$.

\begin{definition}\label{defD}
The function ${\mathcal D}_{{\bld d},{\bld r},{\bld u},k}:I\to{\bf R}$
is defined by
\begin{align*}
{\mathcal D}_{{\bld d},{\bld r},{\bld u},k}(x)
=\frac{1}{q}\sum_{0\le i_{1}<\cdots <i_{|\bld u|}\le k}\sum_{\psi_{\bld u}}
\int_{0}^{x}\prod_{m=1}^{|\bld u|}
\Big(\frac{\Phi_{\psi_{\bld u}(m)}}{r_{\psi_{\bld u}(m)}}
-\frac{\Phi_{q-1}}{r_{q-1}}\Big)\circ\phi^{i_{m}}d\mu_{{\bld d},{\bld r}}.
\end{align*}
\end{definition}
We will give a recursive relation for
${\mathcal D}_{{\bld d},{\bld r},{\bld u},k}(x)$ (see Proposition
\ref{diffprop1} below),
which gives the definition of generalized Takagi functions. 

\begin{lemma}\label{difflem1}
For any $k, \beta \in{\bf N}$ with $\beta+2>k$, 
and integers $l,k$ with $0 \leq l \leq q-1,\ 0 \leq n \leq q^{k}-1$,
we have
$$
{\rm E}_{\mu_{{\bld d},{\bld r}}}
\Big(\Big(\frac{\Phi_{l}}{r_{l}}-\frac{\Phi_{q-1}}{r_{q-1}}\Big)
\circ\phi^{\beta};I_{k}(n)\Big)=0.
$$
\end{lemma}

\begin{proof}
For the $q$-adic representations $n=n_{k-1}\cdots n_{0}$ 
and $qj+\sigma^{-j}(l)=j\sigma^{-j}(l)$, we have 
\begin{align}\label{Ezero}
& {\rm E}_{\mu_{{\bld d},{\bld r}}}
\Big(\Big(\frac{\Phi_{l}}{r_{l}}-\frac{\Phi_{q-1}}{r_{q-1}}\Big)
\circ\phi^{\beta};I_{k}(n_{k-1}\cdots n_{0})\Big)\nonumber\\
& =\frac{1}{r_{l}}\sum_{j=0}^{q-1}
{\rm E}_{\mu_{{\bld d},{\bld r}}}
({\bf 1}_{I_{2}(j\sigma^{-j}(l))}\circ\phi^{\beta};I_{k}(n_{k-1}\cdots n_{0}))\nonumber\\
& \quad -\frac{1}{r_{q-1}}\sum_{j=0}^{q-1}
{\rm E}_{\mu_{{\bld d},{\bld r}}}
({\bf 1}_{I_{2}(j\sigma^{-j}(q-1))}\circ\phi^{\beta};I_{k}(n_{k-1}\cdots n_{0})).
\end{align}
By Lemma \ref{phicom},
\begin{align*}
& \sum_{j=0}^{q-1}{\rm E}_{\mu_{{\bld d},{\bld r}}}
({\bf 1}_{I_{2}(j\sigma^{-j}(l))}\circ\phi^{\beta};I_{k}(n_{k-1}\cdots n_{0}))\\
& ={\rm E}_{\mu_{{\bld d},{\bld r}}}
({\bf 1}_{\bigcup_{0\le b_{2},\ldots,b_{\beta+1}\le q-1} 
\bigcup_{b_{1}=0}^{q-1}
I_{\beta +2}(b_{\beta +1}\cdots b_{2}b_{1}
\sigma^{-b_{1}}(l))};I_{k}(n_{k-1}\cdots n_{0}))\\
& ={\rm E}_{\mu_{{\bld d},{\bld r}}}
({\bf 1}_{\bigcup_{0\le b_{1},\ldots,b_{\beta-k+1}\le q-1} 
I_{\beta +2}(n_{k-1}\cdots n_{0}b_{\beta -k+1}\cdots b_{1}\sigma^{-b_{1}}(l))})\\
& =\sum_{0\le b_{1},\ldots,b_{\beta-k+1}\le q-1} 
\mu_{{\bld d},{\bld r}}(I_{\beta +2}(n_{k-1}\cdots 
n_{0}b_{\beta -k+1}
\cdots b_{1}\sigma^{-b_{1}}(l))).
\end{align*}
Here, by Lemma \ref{phimeasure},
\begin{align*}
& \mu_{{\bld d},{\bld r}}(I_{\beta +2}(n_{k-1}\cdots n_{0}b_{\beta -k+1}
\cdots b_{1}\sigma^{-b_{1}}(l)))\\
& =\mu_{{\bld d},{\bld r}}(I_{\beta +1}(n_{k-1}\cdots n_{0}b_{\beta -k+1}
\cdots b_{1}))\mu_{{\bld r}_{\sigma^{b_{1}}},{\bld r}}(I_{1}(\sigma^{-b_{1}}(l)))\\
& =r_{l}\mu_{{\bld d},{\bld r}}(I_{\beta +1}(n_{k-1}\cdots n_{0}b_{\beta -k+1}\cdots b_{1})),
\end{align*}
and hence
\begin{align}\label{Ereduce}
\sum_{j=0}^{q-1}{\rm E}_{\mu_{{\bld d},{\bld r}}}
({\bf 1}_{I_{2}(j\sigma^{-j}(l))}\circ\phi^{\beta};I_{k}(n_{k-1}\cdots n_{0}))
=r_{l}\mu_{{\bld d},{\bld r}}(I_{k}(n_{k-1}\cdots n_{0})).
\end{align}
Substituting \eqref{Ereduce} into \eqref{Ezero}, we obtain 
the assertion.
\end{proof}

\begin{lemma}\label{difflem2}
For any ${\bldL u}$ with $|{\bldL u}|\ge 2$ and $k\in{\bf N}$, let 
$\{\beta_{m}\}_{m=1}^{|{\bld u}|}$ be a strictly increasing
sequence with $\beta_{1}+2>k$. Then we have 
$$
{\rm E}_{\mu_{{\bld d},{\bld r}}}
\Big(\prod_{m=1}^{|\bld u|}
\Big(\frac{\Phi_{\psi_{\bld u}(m)}}{r_{\psi_{\bld u}(m)}}
-\frac{\Phi_{q-1}}{r_{q-1}}\Big)\circ\phi^{\beta_{m}}
;I_{k}(n)\Big)=0 
$$
for every $0 \leq n \leq q^{k}-1$. 
\end{lemma}

\begin{proof}
Set $\alpha=\psi_{\bld u}(1)$. Then $u_{\alpha}>0$ by the definition of $\psi_{\bld u}$.
We classify the set of $\psi_{\bld u}$'s by $\alpha$.
By the definition of $\psi_{\bld u}$,
\begin{align*}
& \psi_{\bld u}:\{1,2,\ldots,|{\bldL u}|\}\longrightarrow \{0,1,\ldots,q-2\},\\
& \#\{2\le m\le |{\bldL u}|;\psi_{{\bld u}}(m)=j \}
=\begin{cases}
u_{j}-1, & \mbox{if $j=\alpha$},\\
u_{j}, & \mbox{if $j\neq\alpha$},
\end{cases}
\end{align*}
and 
\begin{align*}
& \psi_{{\bld u}-{\bld e}_{\alpha}}:\{1,2,\ldots,|{\bldL u}|-1\}
\longrightarrow \{0,1,\ldots,q-2\},\\
& \#\{1\le m\le |{\bldL u}|-1;\psi_{{\bld u}-{\bld e}_{\alpha}}(m)=j \}
=\begin{cases}
u_{j}-1, & \mbox{if $j=\alpha$},\\
u_{j}, & \mbox{if $j\neq\alpha$}.
\end{cases}
\end{align*}
Hence, for any $\psi_{\bld u}$ there exists a unique 
$\psi_{{\bld u}-{\bld e}_{\alpha}}$ such that
\begin{equation}\label{sift}
\psi_{\bld u}(m)=\psi_{{\bld u}-{\bld e}_{\alpha}}(m-1),\quad 2\le m\le |{\bldL u}|.
\end{equation}
It follows from \eqref{sift} that
\begin{align}
& {\rm E}_{\mu_{{\bld d},{\bld r}}}
\Big(\prod_{m=1}^{|\bld u|}
\Big(\frac{\Phi_{\psi_{\bld u}(m)}}{r_{\psi_{\bld u}(m)}}
-\frac{\Phi_{q-1}}{r_{q-1}}\Big)\circ\phi^{\beta_{m}}
;I_{k}(n)\Big)\nonumber\\
& ={\rm E}_{\mu_{{\bld d},{\bld r}}}
\Big(\Big(\Big(\frac{\Phi_{\psi_{\bld u}(1)}}{r_{\psi_{\bld u}(1)}}
-\frac{\Phi_{q-1}}{r_{q-1}}\Big)\circ\phi^{\beta_{1}}\Big)
\times \Big(\prod_{m=2}^{|\bld u|}
\Big(\frac{\Phi_{\psi_{\bld u}(m)}}{r_{\psi_{\bld u}(m)}}
-\frac{\Phi_{q-1}}{r_{q-1}}\Big)\circ\phi^{\beta_{m}}\Big)
;I_{k}(n)\Big)\nonumber\\
& ={\rm E}_{\mu_{{\bld d},{\bld r}}}
\Big(\Big(\Big(\frac{\Phi_{\alpha}}{r_{\alpha}}
-\frac{\Phi_{q-1}}{r_{q-1}}\Big)\circ\phi^{\beta_{1}}\Big)
\times \Big(\prod_{m=1}^{|\bld u|-1}
\Big(\frac{\Phi_{\psi_{{\bld u}-{\bld e}_{\alpha}}(m)}}
{r_{\psi_{{\bld u}-{\bld e}_{\alpha}}(m)}}
-\frac{\Phi_{q-1}}{r_{q-1}}\Big)\circ\phi^{\beta_{m+1}}\Big)
;I_{k}(n)\Big).\label{expect00}
\end{align}
Here we express $I_{k}(n_{k-1}\cdots n_{0})$, $n=n_{k-1}\cdots n_{0}$, 
as 
\begin{align}\label{expect11}
I_{k}(n_{k-1}\cdots n_{0})=
\bigcup_{0 \leq b_{0},\cdots,b_{\beta_{1}-k+1} \leq q-1}
I_{\beta_{1}+2}(n_{k-1}\cdots n_{0}b_{\beta_{1}-k+1}\cdots b_{0}).
\end{align}
Since $\Big(\frac{\Phi_{\alpha}}{r_{\alpha}} 
-\frac{\Phi_{q-1}}{r_{q-1}}\Big)\circ\phi^{\beta_{1}}$ 
in \eqref{expect00} is 
${\mathcal F}_{\beta_{1}+2}$-measurable (see Lemma \ref{phicom}),
it takes a constant value $C_{b_{\beta_{1}-k+1}\cdots b_{0}}$ on
$I_{\beta_{1}+2}(n_{k-1}\cdots n_{0}b_{\beta_{1}-k+1}\cdots b_{0})$.
Hence, by \eqref{expect00} and \eqref{expect11}, 
\begin{align*}
& {\rm E}_{\mu_{{\bld d},{\bld r}}}
\Big(\prod_{m=1}^{|\bld u|}
\Big(\frac{\Phi_{\psi_{\bld u}(m)}}{r_{\psi_{\bld u}(m)}}
-\frac{\Phi_{q-1}}{r_{q-1}}\Big)\circ\phi^{\beta_{m}}
;I_{k}(n)\Big)\\
& =\sum_{0 \leq b_{0},\cdots,b_{\beta_{1}-k+1} \leq q-1}
   C_{b_{\beta_{1}-k+1}\cdots b_{0}}\\
& \quad\times{\rm E}_{\mu_{{\bld d},{\bld r}}}
\Big(\prod_{m=1}^{|\bld u|-1}
\Big(\frac{\Phi_{\psi_{{\bld u}-{\bld e}_{\alpha}}(m)}}{r_{\psi_{{\bld u}-{\bld e}_{\alpha}}(m)}}
-\frac{\Phi_{q-1}}{r_{q-1}}\Big)\circ\phi^{\beta_{m+1}}
;I_{\beta_{1}+2}(n_{k-1}\cdots n_{0}b_{\beta_{1}-k+1}\cdots b_{0})\Big).
\end{align*}
By repeating this $|{\bldL u}|-1$ times, there exists
an integer $l$ with $0\le l\le q-2$ such that
$
{\rm E}_{\mu_{{\bld d},{\bld r}}}
\Big(\prod_{m=1}^{|\bld u|}
\Big(\frac{\Phi_{\psi_{\bld u}(m)}}{r_{\psi_{\bld u}(m)}}
-\frac{\Phi_{q-1}}{r_{q-1}}\Big)\circ\phi^{\beta_{m}}
;I_{k}(n)\Big)
$
is a linear combination of 
$
{\rm E}_{\mu_{{\bld d},{\bld r}}}
\Big(\Big(\frac{\Phi_{l}}{r_{l}}
-\frac{\Phi_{q-1}}{r_{q-1}}\Big)\circ\phi^{\beta_{|{\bld u}|}}
;I_{\beta_{|\bld u|-1}+2}(n')\Big)
$
over $n'$'s. Therefore the assertion follows from Lemma \ref{difflem1}.
\end{proof}

By Lemmas \ref{difflem1} and \ref{difflem2} with $k=1$, 
we have easily the following.

\begin{lemma}\label{difflem3}
For any ${\bldL u}$ with $|{\bldL u}|\ge 1$ and
$\{\beta_{m}\}_{m=1}^{|{\bld u}|}$ with 
$0\le \beta_{1}<\beta_{2}<\cdots<\beta_{|{\bld u}|}$, we have
$$
{\rm E}_{\mu_{{\bld d},{\bld r}}}
\Big(\prod_{m=1}^{|\bld u|}
\Big(\frac{\Phi_{\psi_{\bld u}(m)}}{r_{\psi_{\bld u}(m)}}
-\frac{\Phi_{q-1}}{r_{q-1}}\Big)\circ\phi^{\beta_{m}}\Big)=0.
$$
\end{lemma}

\begin{proposition}\label{diffprop1}
%
{\rm (i)} If ${\bldL u}={\bldL e}_{l}$, then
\begin{align*}
{\mathcal D}_{{\bld d},{\bld r},{\bld e}_{l},k}(x)
=\frac{1}{q}\sum_{j=0}^{k}\sum_{n=0}^{q^{j}-1}
\mu_{{\bld d},{\bld r}}(I_{j}(n)){\bf 1}_{I_{j}(n)}(x)
\int_{0}^{\phi^{j}(x)}\Big(\frac{\Phi_{l}}{r_{l}}
-\frac{\Phi_{q-1}}{r_{q-1}}\Big)d\mu_{{\bld r}_{\sigma^{n}},{\bld r}}.
\end{align*}
\noindent{\rm (ii)} If $|{\bldL u}|\ge 2$, then
\begin{align*}
{\mathcal D}_{{\bld d},{\bld r},{\bld u},k}(x)
=& \sum_{j=0}^{k-|{\bld u}|+1}
\sum_{\stackrel{\scriptstyle{\alpha=0}}{u_{\alpha}>0}}^{q-2}
\Big(\Big(\frac{\Phi_{\alpha}}{r_{\alpha}}
-\frac{\Phi_{q-1}}{r_{q-1}}\Big)\circ\phi^{j}(x)\Big)\\
& \times\sum_{n=0}^{q^{j+1}-1}
\mu_{{\bld d},{\bld r}}(I_{j+1}(n)){\bf 1}_{I_{j+1}(n)}(x)
\Big({\mathcal D}_{{\bld r}_{\sigma^{n}},{\bld r},{\bld u}
-{\bld e}_{\alpha},k-j-1}
\circ\phi^{j+1}(x)\Big).
\end{align*}
\end{proposition}

\begin{proof}
Taking ${\bldL u}={\bldL e}_{l}$ in Definition \ref{defD}, we have
\begin{align*}
{\mathcal D}_{{\bld d},{\bld r},{\bld e}_{l},k}(x)
=\frac{1}{q}\sum_{j=0}^{k}
\int_{0}^{x}\Big(\frac{\Phi_{l}}{r_{l}}
-\frac{\Phi_{q-1}}{r_{q-1}}\Big)\circ\phi^{j}d\mu_{{\bld d},{\bld r}}.
\end{align*}
If $x \in I_{j}(n)$, then, by Lemmas \ref{difflem1} and \ref{phiEx},
\begin{eqnarray*}
\int_{0}^{x}\Big(\frac{\Phi_{l}}{r_{l}}
-\frac{\Phi_{q-1}}{r_{q-1}}\Big)\circ\phi^{j}d\mu_{{\bld d},{\bld r}}
&=& {\rm E}_{\mu_{{\bld d},{\bld r}}}
\Big(\Big(\frac{\Phi_{l}}{r_{l}}-\frac{\Phi_{q-1}}{r_{q-1}}\Big)
\circ\phi^{j};I_{j}(n)\cap[0,x]\Big)\\
&=& \mu_{{\bld d},{\bld r}}(I_{j}(n))
\int_{0}^{\phi^{j}(x)}\Big(\frac{\Phi_{l}}{r_{l}}
-\frac{\Phi_{q-1}}{r_{q-1}}\Big)d\mu_{{\bld r}_{\sigma^{n}},{\bld r}},
\end{eqnarray*}
which gives (i).

We express the sum $\sum_{0\le i_{1}<\cdots <i_{|\bld u|}\le k}$
in Definition \ref{defD} as
$
\sum_{j=0}^{k-|{\bld u}|+1}
\sum_{j+1\le i_{2}<\cdots <i_{|\bld u|}\le k},
$
then, set $i'_{m-1}=i_{m}-j-1$. Then we have, by \eqref{sift},
\begin{eqnarray*}
{\mathcal D}_{{\bld d},{\bld r},{\bld u},k}(x)
&=& \frac{1}{q}\sum_{j=0}^{k-|{\bld u}|+1}
\sum_{0\le i'_{1}<\cdots <i'_{|\bld u|-1}\le k-j-1}\sum_{\psi_{\bld u}}\\
& & 
\int_{0}^{x}\Big(\Big(
\frac{\Phi_{\psi_{\bld u}(1)}}{r_{\psi_{\bld u}(1)}}
-\frac{\Phi_{q-1}}{r_{q-1}}\Big)\circ\phi^{j}\Big)\times
\Big(\prod_{m=2}^{|\bld u|}
\Big(\frac{\Phi_{\psi_{\bld u}(m)}}{r_{\psi_{\bld u}(m)}}
-\frac{\Phi_{q-1}}{r_{q-1}}\Big)\circ\phi^{i'_{m-1}+j+1}\Big)
d\mu_{{\bld d},{\bld r}}\\
&=& \frac{1}{q}\sum_{j=0}^{k-|{\bld u}|+1}
\sum_{0\le i'_{1}<\cdots <i'_{|\bld u|-1}\le k-j-1}
\sum_{\stackrel{\scriptstyle{\alpha=0}}{u_{\alpha}>0}}^{q-2}
\sum_{\psi_{{\bld u}-{\bld e}_{\alpha}}}\\
& & 
\int_{0}^{x}\Big(\Big(
\frac{\Phi_{\alpha}}{r_{\alpha}}
-\frac{\Phi_{q-1}}{r_{q-1}}\Big)\circ\phi^{j}\Big)\times
\Big(\prod_{m=1}^{|\bld u|-1}
\Big(\frac{\Phi_{\psi_{{\bld u}-{\bld e}_{\alpha}}(m)}}{r_{\psi_{{\bld u}-{\bld e}_{\alpha}}(m)}}
-\frac{\Phi_{q-1}}{r_{q-1}}\Big)\circ\phi^{i'_{m}+j+1}\Big)
d\mu_{{\bld d},{\bld r}}.
\end{eqnarray*}
By Lemma \ref{phicom},
$\Big(\frac{\Phi_{\alpha}}{r_{\alpha}}
-\frac{\Phi_{q-1}}{r_{q-1}}\Big)\circ\phi^{j}$ is 
${\mathcal F}_{j+2}$-measurable. 
From \eqref{kinji00} and Lemmas \ref{difflem1} and \ref{difflem2}, 
it follows that  
$$
{\rm E}_{\mu_{{\bld d},{\bld r}}}
\Big(\prod_{m=1}^{|\bld u|-1}
\Big(\frac{\Phi_{\psi_{{\bld u}-{\bld e}_{\alpha}}(m)}}{r_{\psi_{{\bld u}-{\bld e}_{\alpha}}(m)}}
-\frac{\Phi_{q-1}}{r_{q-1}}\Big)\circ\phi^{i'_{m}+j+1}\Big|\mathcal{F}_{j+2}\Big)=0.
$$
Hence we have, by Lemma \ref{bubun},
\begin{eqnarray}\label{Eeq1}
{\mathcal D}_{{\bld d},{\bld r},{\bld u},k}(x)
&=& \frac{1}{q}\sum_{j=0}^{k-|{\bld u}|+1}
\sum_{0\le i'_{1}<\cdots <i'_{|\bld u|-1}\le k-j-1}
\sum_{\stackrel{\scriptstyle{\alpha=0}}{u_{\alpha}>0}}^{q-2}
\sum_{\psi_{{\bld u}-{\bld e}_{\alpha}}}
\Big(\Big(\frac{\Phi_{\alpha}}{r_{\alpha}}
-\frac{\Phi_{q-1}}{r_{q-1}}\Big)\circ\phi^{j}(x)\Big)\nonumber\\
& & \times{\rm E}_{\mu_{{\bld d},{\bld r}}}
\Big(\prod_{m=1}^{|\bld u|-1}
\Big(\frac{\Phi_{\psi_{{\bld u}-{\bld e}_{\alpha}}(m)}}{r_{\psi_{{\bld u}-{\bld e}_{\alpha}}(m)}}
-\frac{\Phi_{q-1}}{r_{q-1}}\Big)\circ\phi^{i'_{m}}\circ\phi^{j+1};[0,x]\Big).
\end{eqnarray}
Lemmas \ref{phiE} and \ref{difflem3} give, for every $l$, 
\begin{align}\label{Eeq2}
& {\rm E}_{\mu_{{\bld d},{\bld r}}}
\Big(\prod_{m=1}^{|\bld u|-1}
\Big(\frac{\Phi_{\psi_{{\bld u}-{\bld e}_{\alpha}}(m)}}{r_{\psi_{{\bld u}-{\bld e}_{\alpha}}(m)}}
-\frac{\Phi_{q-1}}{r_{q-1}}\Big)\circ\phi^{i'_{m}}\circ\phi^{j+1};I_{j+1}(l)\Big)\nonumber\\
& = \mu_{{\bld d},{\bld r}}(I_{j+1}(l))
{\rm E}_{\mu_{{\bld r}_{\sigma^{l}},{\bld r}}}
\Big(\prod_{m=1}^{|\bld u|-1}
\Big(\frac{\Phi_{\psi_{{\bld u}-{\bld e}_{\alpha}}(m)}}{r_{\psi_{{\bld u}-{\bld e}_{\alpha}}(m)}}
-\frac{\Phi_{q-1}}{r_{q-1}}\Big)\circ\phi^{i'_{m}}\Big)=0.
\end{align}
Lemma \ref{phiEx} gives, for $x\in I_{j+1}(n)$, 
\begin{align}\label{Eeq3}
& {\rm E}_{\mu_{{\bld d},{\bld r}}}
\Big(\prod_{m=1}^{|\bld u|-1}
\Big(\frac{\Phi_{\psi_{{\bld u}-{\bld e}_{\alpha}}(m)}}{r_{\psi_{{\bld u}-{\bld e}_{\alpha}}(m)}}
-\frac{\Phi_{q-1}}{r_{q-1}}\Big)\circ\phi^{i'_{m}}\circ\phi^{j+1};I_{j+1}(n)\cap[0,x]\Big)
\nonumber\\
& = \mu_{{\bld d},{\bld r}}(I_{j+1}(n))
{\rm E}_{\mu_{{\bld r}_{\sigma^{n}},{\bld r}}}
\Big(\prod_{m=1}^{|\bld u|-1}
\Big(\frac{\Phi_{\psi_{{\bld u}-{\bld e}_{\alpha}}(m)}}{r_{\psi_{{\bld u}-{\bld e}_{\alpha}}(m)}}
-\frac{\Phi_{q-1}}{r_{q-1}}\Big)\circ\phi^{i'_{m}};[0,\phi^{j+1}(x)]\Big).
\end{align}
Combining \eqref{Eeq1}, \eqref{Eeq2}, \eqref{Eeq3}, and Definition \ref{defD},
we obtain for $x\in I_{j+1}(n)$ 
\begin{align*}
{\mathcal D}_{{\bld d},{\bld r},{\bld u},k}(x)
=& \sum_{j=0}^{k-|{\bld u}|+1}
\sum_{\stackrel{\scriptstyle{\alpha=0}}{u_{\alpha}>0}}^{q-2}
\Big(\Big(\frac{\Phi_{\alpha}}{r_{\alpha}}
-\frac{\Phi_{q-1}}{r_{q-1}}\Big)\circ\phi^{j}(x)\Big)\\
& \times\mu_{{\bld d},{\bld r}}(I_{j+1}(n))
\Big({\mathcal D}_{{\bld r}_{\sigma^{n}},{\bld r},{\bld u}
-{\bld e}_{\alpha},k-j-1}
\circ\phi^{j+1}(x)\Big),
\end{align*}
which gives (ii).
\end{proof}

\section{Completion of the proof of Theorem  \ref{takagiTH}}\label{Complete00}

\begin{lemma}\label{appro00}
For any ${\bldL u}$ with $|{\bldL u}|\ge 1$, we have
\begin{align*}
\max_{{\bld u}',|{\bld u}'|=|{\bld u}|}
\sup_{{\bld d}'}\|{\mathcal T}_
{{\bld d}',{\bld r},{\bld u}'}\|_{\infty}\le 
\frac{(q-1)^{|{\bld u}|-1}}{q\displaystyle{\max_{0\le a\le q-1}\{r_{a}\}}}
\Big(\frac{2}{\displaystyle{\min_{0\le a\le q-1}\{r_{a}\}}}
\frac{1}{1-\displaystyle{\max_{0\le a\le q-1}\{r_{a}\}}}\Big)^{|{\bld u}|},
\end{align*}
where $\| \cdot \|_{\infty}$ means the supremum norm for $x\in I$.
\end{lemma}

\begin{proof}
Define
${\mathcal M}_{{\bld d},{\bld r},{\bld u},j}(x)$
by 
\begin{align*}
{\mathcal M}_{{\bld d},{\bld r},{\bld u},j}(x)=
\begin{cases}
\displaystyle{\frac{1}{q}\sum_{n=0}^{q^{j}-1}
\mu_{{\bld d},{\bld r}}(I_{j}(n)){\bf 1}_{I_{j}(n)}(x)
\int_{0}^{\phi^{j}(x)}\Big|\frac{\Phi_{l}}{r_{l}}
-\frac{\Phi_{q-1}}{r_{q-1}}\Big|d\mu_{{\bld r}_{\sigma^{n}},{\bld r}}},
& {\bldL u}={\bldL e}_{l},\\
\displaystyle{
\sum_{\stackrel{\scriptstyle{\alpha=0}}{u_{\alpha}>0}}^{q-2}
\Big|\Big(\frac{\Phi_{\alpha}}{r_{\alpha}}
-\frac{\Phi_{q-1}}{r_{q-1}}\Big)\circ\phi^{j}(x)\Big|}\\
\displaystyle{
\times\sum_{n=0}^{q^{j+1}-1}
\mu_{{\bld d},{\bld r}}(I_{j+1}(n)){\bf 1}_{I_{j+1}(n)}(x)
|{\mathcal T}_{{\bld r}_{\sigma^{n}},{\bld r},{\bld u}-{\bld e}_{\alpha}}
\circ\phi^{j+1}(x)|}, & |{\bldL u}|\ge 2.
\end{cases}
\end{align*}
Since $|\Phi_{l}(x)|\le 1$, we have for $j\in{\bf N}\cup\{0\}$ 
\begin{equation}\label{bound00}
\Big|\Big(\frac{\Phi_{l}}{r_{l}}
-\frac{\Phi_{q-1}}{r_{q-1}}\Big)\circ\phi^{j}(x)\Big|
\le \frac{2}{\displaystyle{\min_{0\le a\le q-1}\{r_{a}\}}}.
\end{equation}
Fix $x\in I$. 
For every $j$, there exits an $m_{j}$ such that $x \in I_{j}(m_{j})$.
Then, by \eqref{bound00}, 
\begin{align*}
{\mathcal M}_{{\bld d},{\bld r},{\bld e}_{l},j}(x)
& \le \frac{1}{q}\mu_{{\bld d},{\bld r}}(I_{j}(m_{j}))
\frac{2}{\displaystyle{\min_{0\le a\le q-1}\{r_{a}\}}}
\mu_{{\bld r}_{\sigma^{m_{j}}},{\bld r}}([0,\phi^{j}(x)])\\
& \le 
\frac{2}{\displaystyle{q\min_{0\le a\le q-1}\{r_{a}\}}}
(\max_{0\le a\le q-1}\{d_{a}\})
(\max_{0\le a\le q-1}\{r_{a}\})^{j-1}\\
& \le 
\frac{2}{\displaystyle{q\min_{0\le a\le q-1}\{r_{a}\}}}
(\max_{0\le a\le q-1}\{r_{a}\})^{j-1}.
\end{align*}
Hence
\begin{align}\label{norm0}
\sum_{j=0}^{\infty}\|
{\mathcal M}_{{\bld d},{\bld r},{\bld e}_{l},j}\|_{\infty}
\le\frac{2}{\displaystyle{q\min_{0\le a\le q-1}\{r_{a}\}}}
\frac{(\displaystyle{\max_{0\le a\le q-1}\{r_{a}\}})^{-1}}
{1-\displaystyle{\max_{0\le a\le q-1}\{r_{a}\}}}.
\end{align}
From $\|{\mathcal T}_{{\bld d},{\bld r},{\bld e}_{l}}\|_{\infty}\le
\sum_{j=0}^{\infty}\|
{\mathcal M}_{{\bld d},{\bld r},{\bld e}_{l},j}\|_{\infty}$
and \eqref{norm0}, it follows that
\begin{align}\label{estimate0a}
\sup_{{\bld d}'}\|{\mathcal T}_
{{\bld d}',{\bld r},{\bld e}_{l}}\|_{\infty}\le 
\frac{2}{\displaystyle{q\min_{0\le a\le q-1}\{r_{a}\}}}
\frac{(\displaystyle{\max_{0\le a\le q-1}\{r_{a}\}})^{-1}}
{1-\displaystyle{\max_{0\le a\le q-1}\{r_{a}\}}}.
\end{align}
Fix $x\in I$. 
For every $j$, there exits an $m_{j}$ such that $x \in I_{j+1}(m_{j})$.
Then, by \eqref{bound00}, we have for $|{\bldL u}|\ge 2$
\begin{align*}
{\mathcal M}_{{\bld d},{\bld r},{\bld u},j}(x)
& \le \frac{2}{\displaystyle{\min_{0\le a\le q-1}\{r_{a}\}}}
\mu_{{\bld d},{\bld r}}(I_{j+1}(m_{j}))
\sum_{\stackrel{\scriptstyle{\alpha=0}}{u_{\alpha}>0}}^{q-2}
|{\mathcal T}_{{\bld r}_{\sigma^{m_{j}}},{\bld r},{\bld u}
-{\bld e}_{\alpha}}
\circ\phi^{j+1}(x)|\\
& \le \frac{2(q-1)}{\displaystyle{\min_{0\le a\le q-1}\{r_{a}\}}}
(\max_{0\le a\le q-1}\{d_{a}\})
(\max_{0\le a\le q-1}\{r_{a}\})^{j}
\max_{{\bld u}',|{\bld u}'|=|{\bld u}|-1}
|{\mathcal T}_{{\bld r}_{\sigma^{m_{j}}},{\bld r},{\bld u}'}
\circ\phi^{j+1}(x)|\\
& \le \frac{2(q-1)}{\displaystyle{\min_{0\le a\le q-1}\{r_{a}\}}}
(\max_{0\le a\le q-1}\{r_{a}\})^{j}
\max_{{\bld u}',|{\bld u}'|=|{\bld u}|-1}\sup_{{\bld d}'}
\|{\mathcal T}_{{\bld d}',{\bld r},{\bld u}'}\|_{\infty}.
\end{align*}
Hence
\begin{align}\label{norm1}
\sum_{j=0}^{\infty}\|
{\mathcal M}_{{\bld d},{\bld r},{\bld u},j}\|_{\infty}
\le  \frac{2(q-1)}{\displaystyle{\min_{0\le a\le q-1}\{r_{a}\}}}
\frac{1}{1-\displaystyle{\max_{0\le a\le q-1}\{r_{a}\}}}
\max_{{\bld u}',|{\bld u}'|=|{\bld u}|-1}
\sup_{{\bld d}'}
\|{\mathcal T}_{{\bld d}',{\bld r},{\bld u}'}\|_{\infty}.
\end{align}
From $\|{\mathcal T}_{{\bld d},{\bld r},{\bld u}}\|_{\infty}\le
\sum_{j=0}^{\infty}\|
{\mathcal M}_{{\bld d},{\bld r},{\bld u},j}\|_{\infty}$
and \eqref{norm1}, it follows that
\begin{align*}
\max_{{\bld u}',|{\bld u}'|=|{\bld u}|}
\sup_{{\bld d}'}\|{\mathcal T}_
{{\bld d}',{\bld r},{\bld u}'}\|_{\infty}\le \frac{2(q-1)}{\displaystyle{\min_{0\le a\le q-1}\{r_{a}\}}}
\frac{1}{1-\displaystyle{\max_{0\le a\le q-1}\{r_{a}\}}}
\max_{{\bld u}',|{\bld u}'|=|{\bld u}|-1}
\sup_{{\bld d}'}
\|{\mathcal T}_{{\bld d}',{\bld r},{\bld u}'}\|_{\infty}.
\end{align*}
By repeating this $|{\bldL u}|-1$ times,
there exists an integer $l$ with $0\le l\le q-1$ such that
\begin{align}\label{norm3}
\max_{{\bld u}',|{\bld u}'|=|{\bld u}|}
\sup_{{\bld d}'}\|{\mathcal T}_
{{\bld d}',{\bld r},{\bld u}'}\|_{\infty}\le 
\Big(\frac{2(q-1)}{\displaystyle{\min_{0\le a\le q-1}\{r_{a}\}}}
\frac{1}{1-\displaystyle{\max_{0\le a\le q-1}\{r_{a}\}}}\Big)^{|{\bld u}|-1}
\sup_{{\bld d}'}\|{\mathcal T}_{{\bld d}',{\bld r},{\bld e}_{l}}\|_{\infty}.
\end{align}
Combining \eqref{norm3} with \eqref{estimate0a}, we obtain
the assertion.
\end{proof}

\begin{lemma}\label{last00}
For any ${\bldL u}$ with $|{\bldL u}|\ge 1$, we have
\begin{align*}
& \max_{{\bld u}',|{\bld u}'|=|{\bld u}|}\sup_{{\bld d}'}
\Big\|{\mathcal T}_{{\bld d}',{\bld r},{\bld u}'}
-{\mathcal D}_{{\bld d}',{\bld r},{\bld u}',k}\Big\|_{\infty}
\le P_{|{\bld u}|-1}(k)(\max_{0\le a\le q-1}\{r_{a}\})^{k},
\end{align*}
where $P_{|{\bld u}|-1}(k)$ is a polynomial of $k$ 
with degree $|{\bldL u}|-1$. 
\end{lemma}

\begin{proof}
We prove this by induction on $|{\bldL u}|$. 
By the same argument as in the proof of Lemma \ref{appro00},
\begin{align*}
|{\mathcal T}_{{\bld d},{\bld r},{\bld e}_{l}}(x)
-{\mathcal D}_{{\bld d},{\bld r},{\bld e}_{l},k}(x)|
\le \sum_{j=k+1}^{\infty}
|{\mathcal M}_{{\bld d},{\bld r},{\bld e}_{l},j}(x)|
\le 
\frac{2}{\displaystyle{q\min_{0\le a\le q-1}\{r_{a}\}}}
\frac{(\displaystyle{\max_{0\le a\le q-1}\{r_{a}\}})^{k}}
{1-\displaystyle{\max_{0\le a\le q-1}\{r_{a}\}}}.
\end{align*}
Hence
\begin{align*}
\sup_{{\bld d}'}\|{\mathcal T}_{{\bld d}',{\bld r},{\bld e}_{l}}
-{\mathcal D}_{{\bld d}',{\bld r},{\bld e}_{l},k}\|_{\infty}
\le 
\frac{2}{\displaystyle{q\min_{0\le a\le q-1}\{r_{a}\}}}
\frac{(\displaystyle{\max_{0\le a\le q-1}\{r_{a}\}})^{k}}
{1-\displaystyle{\max_{0\le a\le q-1}\{r_{a}\}}}.
\end{align*}
Fix $x\in I$. 
For every $j$, there exits an $m_{j}$ such that $x \in I_{j+1}(m_{j})$.
Then, by \eqref{bound00}, we have for $|{\bldL u}|\ge 2$
\begin{align*}
& |{\mathcal T}_{{\bld d},{\bld r},{\bld u}}(x)
-{\mathcal D}_{{\bld d},{\bld r},{\bld u},k}(x)|\\
& \le \sum_{j=k-|{\bld u}|+2}^{\infty}
\sum_{\stackrel{\scriptstyle{\alpha=0}}{u_{\alpha}>0}}^{q-2}
\Big|\Big(\frac{\Phi_{\alpha}}{r_{\alpha}}
-\frac{\Phi_{q-1}}{r_{q-1}}\Big)\circ\phi^{j}(x)\Big|\\
& \qquad \times\sum_{n=0}^{q^{j+1}-1}
\mu_{{\bld d},{\bld r}}(I_{j+1}(n)){\bf 1}_{I_{j+1}(n)}(x)
\Big|{\mathcal T}_{{\bld r}_{\sigma^{n}},{\bld r},{\bld u}
-{\bld e}_{\alpha}}
\circ\phi^{j+1}(x)\Big|\\
& \quad + \sum_{j=0}^{k-|{\bld u}|+1}
\sum_{\stackrel{\scriptstyle{\alpha=0}}{u_{\alpha}>0}}^{q-2}
\Big|\Big(\frac{\Phi_{\alpha}}{r_{\alpha}}
-\frac{\Phi_{q-1}}{r_{q-1}}\Big)\circ\phi^{j}(x)\Big|\\
& \qquad \times\sum_{n=0}^{q^{j+1}-1}
\mu_{{\bld d},{\bld r}}(I_{j+1}(n)){\bf 1}_{I_{j+1}(n)}(x)
\Big|{
{\mathcal T}_{{\bld r}_{\sigma^{n}},{\bld r},{\bld u}
-{\bld e}_{\alpha}}\circ\phi^{j+1}(x)-
\mathcal D}_{{\bld r}_{\sigma^{n}},{\bld r},{\bld u}
-{\bld e}_{\alpha},k-j-1}
\circ\phi^{j+1}(x)\Big|\\
& \le \sum_{j=k-|{\bld u}|+2}^{\infty}
\frac{2}{\displaystyle{\min_{0\le a\le q-1}\{r_{a}\}}}
(\max_{0\le a\le q-1}\{r_{a}\})^{j}
\sum_{\stackrel{\scriptstyle{\alpha=0}}{u_{\alpha}>0}}^{q-2}
\Big\|{\mathcal T}_{{\bld r}_{\sigma^{m_{j}}},{\bld r},{\bld u}
-{\bld e}_{\alpha}}\Big\|_{\infty}\\
& \quad + \sum_{j=0}^{k-|{\bld u}|+1}
\frac{2}{\displaystyle{\min_{0\le a\le q-1}\{r_{a}\}}}
(\max_{0\le a\le q-1}\{r_{a}\})^{j}
\sum_{\stackrel{\scriptstyle{\alpha=0}}{u_{\alpha}>0}}^{q-2}
\Big\|{\mathcal T}_{{\bld r}_{\sigma^{m_{j}}},{\bld r},{\bld u}
-{\bld e}_{\alpha}}
-{\mathcal D}_{{\bld r}_{\sigma^{m_{j}}},{\bld r},{\bld u}
-{\bld e}_{\alpha},k-j-1}\Big\|_{\infty}\\
& \le \frac{2(q-1)}{\displaystyle{\min_{0\le a\le q-1}\{r_{a}\}}}
\frac{(\displaystyle{\max_{0\le a\le q-1}\{r_{a}\}})^{k-|{\bld u}|+2}}
{1-\displaystyle{\max_{0\le a\le q-1}\{r_{a}\}}}
\max_{{\bld u}',|{\bld u}'|=|{\bld u}|-1}\sup_{{\bld d}'}
\Big\|{\mathcal T}_{{\bld d}',{\bld r},{\bld u}'}\Big\|_{\infty}\\
& \quad + 
\frac{2(q-1)}{\displaystyle{\min_{0\le a\le q-1}\{r_{a}\}}}
\sum_{j=0}^{k-|{\bld u}|+1}
(\max_{0\le a\le q-1}\{r_{a}\})^{j}
\max_{{\bld u}',|{\bld u}'|=|{\bld u}|-1}\sup_{{\bld d}'}
\Big\|{\mathcal T}_{{\bld d}',{\bld r},{\bld u}'}
-{\mathcal D}_{{\bld d}',{\bld r},{\bld u}',k-j-1}\Big\|_{\infty}.
\end{align*}
Hence, by Lemma \ref{appro00} and the assumption of induction,
\begin{align*}
& \max_{{\bld u}',|{\bld u}'|=|{\bld u}|}\sup_{{\bld d}'}
\Big\|{\mathcal T}_{{\bld d}',{\bld r},{\bld u}'}
-{\mathcal D}_{{\bld d}',{\bld r},{\bld u}',k}\Big\|_{\infty}\\
& \ll (\max_{0\le a\le q-1}\{r_{a}\})^{k-|{\bld u}|+2}
+ \sum_{j=0}^{k-|{\bld u}|+1}
(\max_{0\le a\le q-1}\{r_{a}\})^{j}
P_{|{\bld u}|-2}(k-j-1)(\max_{0\le a\le q-1}\{r_{a}\})^{k-j-1}\\
& = (\max_{0\le a\le q-1}\{r_{a}\})^{k-|{\bld u}|+2}
+ (\max_{0\le a\le q-1}\{r_{a}\})^{k-1}
\sum_{j=0}^{k-|{\bld u}|+1}P_{|{\bld u}|-2}(k-j-1),
\end{align*}
where the implied constant depends only on $q$, ${\bldL r}$, 
and $|{\bldL u}|$. Since $\sum_{j=0}^{k-|{\bld u}|+1}
P_{|{\bld u}|-2}(k-j-1)$ is a polynominal of $k$ with 
degree $|{\bldL u}|-1$, we obtain the assertion. 
\end{proof}

By Lemma \ref{last00}, we see that 
\begin{equation*}
\lim_{k\to\infty}{\mathcal D}_{{\bld d},{\bld r},{\bld u},k}(x)
={\mathcal T}_{{\bld d},{\bld r},{\bld u}}(x)
\end{equation*}
holds uniformly for $x\in{I}$.
Thus we obtain Theorem \ref{takagiTH} by 
Propositions \ref{HLsrq} and \ref{diffprop1}.
\qed


\end{document}